\documentclass[11pt]{article}
\usepackage{amsmath,amsthm,amsfonts,latexsym,amssymb,enumerate,color}
\usepackage{graphicx}

\def\CC{\mathbb{C}}

\def\RR{\mathbb{R}}

\def\ZZ{\mathbb{Z}}

\def \codim{\mathop{\rm codim}\nolimits}
\def\dim{\mathop{\rm dim}\nolimits}

\def\ker{\mathop{\rm ker}\nolimits}

\renewcommand\phi{\varphi}

\newcommand{\simast}{%
    \ensuremath{%
       \stackrel{\mathsf{\ast}}{\sim}}}

\newtheorem{thm}{Theorem}[section]
\newtheorem{prop}[thm]{Proposition}
\newtheorem{defn}[thm]{Definition}
\newtheorem{lem}[thm]{Lemma}
\newtheorem{cor}[thm]{Corollary}
\newtheorem{rem}[thm]{Remark}

\numberwithin{equation}{section}

\def\beginpf{\begin{proof}}
\def\endpf{\end{proof}}
\def\beq{\begin{equation}}
\def\eeq{\end{equation}}

\def\imag{\mathop{\rm Im}\nolimits}

\begin{document}

\title{Toeplitz operators and Wiener-Hopf factorisation: an introduction}

\date {}

\author{M.~Cristina C\^amara\thanks{
Center for Mathematical Analysis, Geometry, and Dynamical Systems,
Departamento de Matem\'{a}tica, Instituto Superior T\'ecnico, 1049-001 Lisboa, Portugal. \tt ccamara@math.ist.utl.pt} \ 
}

\maketitle

\begin{abstract}
Wiener-Hopf factorisation plays an important role in the theory of Toeplitz operators. We consider here Toeplitz operators in the Hardy spaces $H^p$ of the upper half-plane and we review how their Fredholm properties can be studied in terms of a Wiener-Hopf factorisation of their symbols, obtaining necssary and sufficient conditions for the operator to be Fredholm or invertible, as well as formulae for their inverses or one-sided inverses when these exist. The results are applied to a class of singular integral equations in $L^p(\RR)$.
\end{abstract}




\section{Introduction}
\label{sec:1}

Wiener-Hopf factorisation, which was developed mainly in connection with the study of singular integral equations, convolution type operators and Riemann-Hilbert problems that are important in a variety of areas in Mathematics, Physics and Engineering, plays a prominent role in the theory of Toeplitz operators. Indeed many spectral properties, Fredholmness, invertibility and formulae for the inverses, when they exist, can be expressed in terms of a Wiener-Hopf (WH for short) factorisation of their symbols.

\vspace {3mm}
Several monographs have appeared in the last four decades (\cite{BSilb, GF, GK,LS, MP, Pross} for instance) intending to present, as systematically and completely as possible, the myriad of results that kept appearing on WH factorisation and its relations with singular integral equations, boundary value problems and Toeplitz operators. However, the enormous quantity of results and the extension of most monographs make it difficult for someone not familiar with this topic to come to grips with the vast amount of information contained there and quickly cut through to get to the heart of the matter.

\vspace {3mm}
The present article is based on a mini-course given in the context of the 13th Advanced Course in Operator Theory and Complex Analysis held in Lyon in 2016. Its aim is to introduce the reader to a set of concepts and results enabling one to rapidly understand the essential ideas behind the notion of WH factorisation and its usefulness in the study of Toeplitz operators. 

\vspace {3mm}
It would be clearly impossible to cover in a single article this vast domain, so the purpose of this paper required choices regarding which results to present, their setting, how to show their interrelations and how they build upon each other. This was done based on the author's own experience in learning the subject, as well as the author's research interests.

\vspace {3mm}
First, it should be explicitly noticed that only scalar-valued symbols will be considered. The factorisation of matrix-valued functions and the study of Toeplitz operators with matrix symbols involve considerably more difficulties that would overshadow the exposition in a first approach.

\vspace {3mm}
Secondly, we study here Toeplitz operators defined in the Hardy spaces of the (upper) half-plane, $H^p(\mathbb C^+)$, with $1<p<\infty$. The existing literature overwhelmingly puts its emphasis on factorisation relative to a closed contour such as the unit circle and on Toeplitz operators defined in the Hardy spaces of the disk, whether in the context of $H^2$ or $H^p\,,\,1<p<\infty$. However, the natural setting in which many problems appear, as for instance those formulated in terms of finite-interval convolution equations, is the real line. Translating the results presented here from the context of the unit circle to that of the real line, for $p\neq 2$, would require considering weighted Hardy spaces of the disk (see \cite {CDR, MP}, for instance).
By considering the half-plane setting from the start, on the one hand, we avoid this; on the other hand, it allows for a better understanding of the reasoning and the techniques involved when the real line is considered, provides a direct approach to problems that are naturally formulated in the context of $\RR$, and allows the use of tools such as the Fourier transform.

 \vspace {3mm} 
The paper is organised as follows. In Section 2 the fundamental function spaces used in the paper and their main properties are described. Toeplitz operators and some of their basic properties
are presented in Section 3, and their relation to paired operators is described in Section 4. In Section 5 the Fredholm properties and invertibility of a Toeplitz operator are characterised in terms of a WH factorisation of its symbol. Piecewise continuous symbols are studied in Section 6, where criteria for existence of a WH factorisation of those functions are presented and explicit formulae for the factors are obtained in a particular case. The results of the previous sections are used in Section 7 to study the spectrum of the singular integral operator $S_J$ in three cases, $J=\mathbb R\,,\,J=\mathbb R^+\,,\,J=[0,1]$, and to obtain an expression for the resolvent operator $(S_{\mathbb R^+}-\lambda\,I_{\mathbb R^+})^{-1}$ with $\lambda\in\mathbb C\setminus[-1,1]$.

 \vspace {3mm} 
 The main references for Section 2 are \cite{Du, duren, O, W}; for Section 3, \cite{BSilb, GK, W}; for Section 4, \cite{BSilb, GGK, Pross}; for Section 5, \cite{Du, MP}; for Sections 6 and 7, \cite{Du, W}.


\section{Spaces of functions}

For $0< p<\infty$, let $L_p$ denote the Lebesgue space of all complex Lebesgue measurable functions $f$ which are $p$-integrable on $\mathbb R$, with the norm
\[
|| f ||_p= (\int_\RR |f(x)|^p dx)^{1/p}.
\]
We denote by $H_p^+$ the Hardy space of all functions $f$ which are analytic in the upper half plane $\mathbb C^+$ such that, for all $y>0$, $|f(x+iy)|^p$ is integrable over $\mathbb R$  and there is a constant $M\in \mathbb R^+$ such that
\[
\int_\RR |f(x+iy)|^p dx<M \,\,\,\,{\rm for \, all}\, \,y>0,
\]
with the norm
\[
|| f ||_{H_p^+}=\sup_{y\in \RR^+}\left (\int_\RR |f(x+iy)|^p dx\right )^{1/p}
\]
(\cite{duren}). If $f_+\in H_p^+$ then the nontangential boundary function
\[
\tilde f_+(x)=\displaystyle{\lim_{z \to x,z\in \CC^+}}f_+(z)
\]
is defined a.e. on $\RR$ and belongs to $L_p$, with $||\tilde f_+||_p=||f_+||_{H_p^+}$.

In what follows, we assume that $1\leq p<\infty$. We have
\begin{equation}
\label{2.1}
\frac{1}{2\pi i}\int_\RR \frac{\tilde f_+(t)}{t-z}dt= \left\{\begin{array}{c}
                          f_+(z),\,\mbox{if}\,\,z\in\CC^+ \\
                             \\
                          0,\,\mbox{if}\,\,z\in\CC^-.
                        \end{array}
 \right..
\end{equation}
Conversely, if $\tilde f_+\in L_p$ and $\int_\RR \frac{\tilde f_+(t)}{t-z}dt=0$ for all $z\in \CC^+$, then $f_+$ defined by 
\[
f_+(z)=\frac{1}{2\pi i}\int_\RR \frac{\tilde f_+(t)}{t-z}dt \,\,\,\mbox{for}\,\,z\in \CC^+
\]
belongs to $H_p^+$ and its boundary value on $\RR$ is $\tilde f_+$.

We define $H_p^-$ similarly for the lower half plane $\CC^-$ and analogous results hold for $f_-\in H_p^-$. In particular, if $\tilde f_-$ is the (nontangential) boundary value of $f_-$ on $\RR$ then
\begin{equation}
\label{2.2}
-\frac{1}{2\pi i}\int_\RR \frac{\tilde f_-(t)}{t-z}dt= \left\{\begin{array}{c}
                          f_-(z),\,\mbox{if}\,\,z\in\CC^- \\
                             \\
                          0,\,\mbox{if}\,\,z\in\CC^+.
                        \end{array}
 \right..
\end{equation}

Identifying each function in $H_p^\pm$ with its boundary value on $\RR$, we have that $H_p^+$ and $H_p^-$ are closed subspaces of $L_p$ and the following holds:
\begin{equation}
\label{2.3}
H_p^+\cap H_p^-=\{0\}
\end{equation}

\begin{equation}
\label{2.4}
\overline{H_p^+}=H_p^-
\end{equation}
and for $p,r\in ]1,\infty[$
\begin{equation}
\label{2.5}
f\in H_p^\pm,\, g\in H_r^\pm\Rightarrow fg\in H_s^\pm\,\,\,\mbox{with}\,\,\frac{1}{s} =\frac{1}{p}+\frac{1}{r}\,.
\end{equation}
By $L_\infty$ we denote the space of all essentially bounded functions on $\RR$, with the norm
\[
||f||_\infty=\operatorname{ess}\sup\limits_{x\in\RR}\,|f(x)|\,,
\]
and by $H_\infty^\pm$ the space of all functions which are analytic and bounded in $\CC^\pm$. We identify $H_\infty^\pm$ with the subspaces of $L_\infty$ consisting of their (nontangential) boundary functions on $\RR$. We have
\begin{equation}
\label{2.6}
a_\pm\in H_\infty^\pm\,,\,f_\pm \in H_p^\pm\Rightarrow a_\pm f_\pm \in H_p^\pm\,\,\,,\,\,\,\mbox {for} \,\,\,1\leq p\leq \infty\,,
\end{equation}
\begin{equation}
\label{2.7}
H_\infty^+\cap H_\infty^-=\CC\,.
\end{equation}
We also define, for $1\leq p<\infty$\,,
\begin{equation}
\label{2.8}
\mathcal H_p^\pm=\lambda_\pm H_p^\pm\,\,,\,\,\,\mbox{where}\,\,\lambda_\pm(x)=x\pm i
\end{equation}
(\cite{W}) and we have
\begin{equation}
\label{2.9}
H_p^\pm \subset \mathcal H_p^\pm\,\,,\,\,   H_p^\pm\subset \mathcal H_1^\pm\,\,,\,\,     H_\infty^\pm\subset \mathcal H_p^\pm\,\,\mbox {for\,\,all}\,\,p>1\,;
\end{equation}
\begin{equation}
\label{2.10}
\overline{\mathcal H_p^+}=\mathcal H_p^-\,\,\mbox {for\,\,all}\,\,p\geq 1\,;
\end{equation}
\begin{equation}
\label{2.11}
\mathcal H_p^+ \cap \mathcal H_p^-=\CC \,\,\mbox {for\,\,all}\,\,p>1\,;
\end{equation}
\begin{equation}
\label{2.12}
\mathcal H_1^+ \cap \mathcal H_1^-=\{0\}\,\,,\,\,\lambda_+ \mathcal H_1^+\cap\lambda_- \mathcal H_1^-=\CC\,;
\end{equation}
\begin{equation}
\label{2.13}
H_p^\pm \mathcal H_{p'}^\pm \subset \mathcal H_1^\pm\,\,\,\mbox{where}\,\,\frac{1}{p}+\frac{1}{p'}=1\,,\,p>1.
\end{equation}

By the Luzin-Privalov theorem, if $f$ is meromorphic in $\CC^+$ (resp. $\CC^-$) and has nontangential boundary value $0$ on a set $E\subset \RR$ with positive measure, then $f(z)=0$ for all $z\in \CC^+$ (resp. $\CC^-$). Thus if a function $f$ belonging to $ H_p^\pm$ or to $\mathcal H_p^\pm$ ($1\leq p \leq \infty$) vanishes on a set of positive measure $E\subset \RR$ then $f$ is identically zero.

\vspace{5mm}
By $C (\RR_\infty)$ we denote the space of all continuous functions on $\RR_\infty=\RR \cup \{\infty\}$, with the supremum norm, and by $\mathcal R$ the set of all rational functions without poles on $\RR_\infty$. The set of all piecewise continuous functions on $\RR_\infty$, i.e., $f\in L_\infty (\RR)$ such that $f$ is continuous on $\RR_\infty$ with the possible exception of finitely many points and with finite limits $f(\pm\infty)\,,\,f(x^\pm)$ for all $x\in \RR$, is denoted by $PC(\RR_\infty)$.

\vspace{5mm}
If $\mathcal A$ is an algebra, we denote by $\mathcal G\mathcal A$ the group of all invertible elements in $\mathcal A$.


\section{Toeplitz operators in $H_p^+$}

Let now $p\in ]1,\infty[$. For $f \in L_p$, the equality

\beq
S_\RR\,f(x)=\frac{1}{i\pi}\int_\RR \frac{f(t)}{t-x} dt,
\eeq
where the integral is understood as a Cauchy principal value, defines, by a classical theorem of M. Riesz (\cite{R}), a bounded operator in $L_p$. We have
\beq
\label{3.1}
S_\RR^*=S_\RR\,\,,\,\,S_\RR^2=I
\eeq
\beq
\label{3.2}
\overline{S_\RR f}=-S_\RR\bar f\,\,,\,\,f\in L_p
\eeq
\beq
\label{3.3}
S_\RR a-aS_\RR\,\, \mbox{is compact if}\,\,a\in C(\RR_\infty).
\eeq
We associate with $S_\RR$ two complementary projections
\beq
\label{3.4}
P^\pm=\frac{I\pm S_\RR}{2}.
\eeq
By the Sokhotski-Plemelj formulae, $P^\pm f$ are the nontangential boundary values on $\RR$ of the analytic functions
\beq
\pm\frac{1}{2\pi i}\int_\RR \frac{f(t)}{t-z} dt\,\,,\,\,z\in \CC^\pm\,,
\eeq
respectively. Thus $P^\pm L_p=H_p^\pm$ and $\ker P^\pm=H_p^\mp$. In terms of these projections, \eqref{2.6} can be expressed by
\beq
\label{3.4A}
P^\pm a_\pm P^\pm=a_\pm P^\pm\,\,,\,\,\mbox{if}\,\, a_\pm \in H_\infty^\pm\,.
\eeq
It follows from \eqref{3.2} and \eqref{2.4} that, if $C$ denotes complex conjugation $(C\,f=\bar f)$ then
\beq
\label{3.5}
C P^\pm = P^\mp C\,\,,\,\,\mbox{i.e.},\,\, C\,P^\pm C=P^\pm
\eeq
and, from \eqref {3.3}, that $P^\pm a - aP^\pm$ is compact in $L_p$ if $a\in C(\RR_\infty)$.

\vspace{5mm}
The Toeplitz operator with symbol $g\in L_\infty$, $T_g$,  is defined in $H_p^+$ by
\beq
\label{4.1}
T_g\,f_+=P_+g\,f_+\,\,\,\,,\,\,f_+\in H_p^+.
\eeq
Whenever we want to make the domain $H_p^+$ explicit, we will use the notation $T_{g,p}$. We have $||g||_\infty \leq ||T_g|| \leq C\,||g||_\infty$ (\cite{Du}), where $C=1$ if $p=2$. Thus $T_g=0$ if and only if $g=0$.

The following properties hold:
\begin{eqnarray}
&&\mbox{(P1)}\,\,\,\,\,\,(T_{g,p})^*=T_{\bar g,p'}\,\,\,\,\,\mbox{where}\,\,\frac{1}{p}+\frac{1}{p'}=1\,;
\nonumber\\
&&\mbox{(P2)}\,\,\,\,\,\,\mbox{if}\,\, a_\pm \in H_\infty^\pm\,\,\mbox{and} \,\,g\in L_\infty,\,\,\mbox{then}\,\, T_{a_-}T_g=T_{ga_-}\,\,,\,\,T_gT_{a_+}=T_{ga_+};
\nonumber\\
&&\mbox{(P3)}\,\,\,\,\,\,\mbox{if}\,\, a_\pm \in \mathcal G\,H_\infty^\pm\,\,\mbox{then} \,\,T_{a_\pm}\,\,\mbox{is invertible and}\,\, (T_{a_\pm})^{-1}=T_{(a_\pm)^{-1}}.
\nonumber
\end{eqnarray}
Note that in general $(T_g)^{-1}\neq T_{g^{-1}}$.

\vspace{5mm}
For $f_+\in H_p^+$ we have that
\beq
\label{4.1A}
f_+\in \ker T_g\Leftrightarrow g\,f_+=f_-\in H_p^-
\eeq
where the equality holds on $\RR$.
Defining $r$ by
\begin{equation}\label{2.15}
r(x)=\frac{x-i}{x+ i}\;,\;\;x\in \RR,
\end{equation}
it follows from \eqref{4.1A} that, for $k\geq 0$,
\beq
\label{4.1B}
\ker T_{r^k}=\{0\}\,\,\,,\,\,\dim\ker T_{r^{-k}}=k.
\eeq
On the other hand from (P2) we have, for $g\in L_\infty\,,\,k\in \ZZ$,
\beq
\label{4.2}
 T_{gr^k}=T_gT_{r^k}\,\,,\,\,\,\mbox{if}\,\,k\geq 0
\eeq
\beq
\label{4.3}
 T_{gr^k}=T_{r^k}T_g\,\,,\,\,\,\mbox{if}\,\,k\leq 0. 
\eeq
From \eqref{4.1A}-\eqref{4.3} we conclude the following.

\begin{prop}
\label{Prop 4.A}
Let $k\in \ZZ$ and let $r$ be given by \eqref{2.15}.

(i) $T_{r^{k}}$ is left invertible if $k\geq 0$, with left inverse $(T_{r^{k}})_l^{-1}=T_{r^{-k}}$; $\dim\ker T_{r^{k}}=0\,,\,\dim\ker T_{r^{k}}^*=k$.

(ii) $T_{r^{k}}$ is right invertible if $k\leq 0$, with right inverse $(T_{r^{k}})_r^{-1}=T_{r^{-k}}$; $\dim\ker T_{r^{k}}=|k|\,,\,\dim\ker T_{r^{k}}^*=0$.

(iii) $T_{r^{k}}$ is Fredholm with index $-k$, and it is invertible if and only if $k=0$.
\end{prop}
Recall that an operator $T$ is Fredholm if and only if its kernel and the kernel of its adjoint are finite dimensional, and the range of $T$ is closed; the Fredholm index of $T$ is $Ind\, T=\dim \ker T-\dim \ker T^*$. We will see in the next section that, in a certain sense, all Toeplitz operators which are Fredholm can be reduced, through an appropriate factorisation of their symbol, to a Toeplitz operator with a simple rational symbol such as the one considered in Proposition \ref{Prop 4.A}. From the latter we see moreover that: (i) either $\dim\ker T_{r^k}=0$ or $\dim\ker T_{r^k}^*=0$\,; (ii) $T_{r^k}$ is always one sided invertible. We will also show in the next theorem that property (i) is shared by all non-zero Toeplitz operators, while (ii) holds for all Toeplitz operators which are Fredholm.

We have:
\begin{thm}
\label{Prop 4.B}
(Coburn's Lemma) If $g\in L_\infty\,,\,g\neq 0$, then either $\ker T_g=\{0\}$ (in $H_p^+$) or $\ker T_g^*=\{0\}$ (in $H_{p'}^+$).
\end{thm}

\begin{proof}
Suppose that there are $f_+\in H_p^+\,,\,h_+\in H_{p'}^+$, different from $0$, such that $f_+\in \ker T_g\,,\,h_+\in \ker T_g^*=\ker T_{\bar g}$. This means that there are $f_-\in H_p^-\setminus \{0\}\,,\,h_-\in H_{p'}^-\setminus \{0\}$ such that
\beq
\nonumber
gf_+=f_-\,\,\,,\,\bar g h_+=h_-.
\eeq
Thus
\beq
\label{4.4}
 f_-\overline{h_+}=gf_+\overline{h_+}=f_+g\overline{h_+}=f_+\overline{h_-}.
\eeq
Since the left hand side of \eqref{4.4} represents a function in $H_1^-$ and the right hand side represents a function in $H_1^+$, we conclude that both are zero. By the Luzin-Privalov theorem either $f_+$ or $\overline{h_-}$ must be zero, which is impossible since $f_+\,,\,h_+\neq 0$.
\end{proof}
Coburn's Lemma can also be stated as follows: a non-zero Toeplitz operator has a trivial kernel or a dense range.

\vspace{5mm}
A necessary and sufficient condition for an operator $A\in \mathcal L(X,Y)$, where $X$ and $Y$ are Banach spaces, to have a left (resp. right) inverse is that $\ker A=\{0\}$ and $\mbox{Im}\, A$ is closed and complemented in $Y$ (resp., $\mbox{Im}\, A=Y$ and $\ker A$ is complemented in $X$) (\cite {GGK}). Therefore we have:

\begin{cor}
\label{Cor 4.C}
Let $T_g$ be a Toeplitz operator in $H_2^+$. If $\mbox{Im}\,T_g$ is closed, then $T_g$ is (at least) one sided invertible.
\end{cor}

\begin{cor}
\label{Cor 4.D}
$T_g$ is invertible if and only if it is Fredholm with index $0$.
\end{cor}

By the Hartman-Wintner theorem (\cite {BSilb}), if $T_g$ is semi-Fredholm, i.e., $\mbox{Im}\,T_g$ is closed and $\ker T_g$ or $\ker T_g^*$ are finite-dimensional, then $g\in \mathcal G L_\infty$. Therefore we also have:

\begin{cor}\label{cor3.5}
\label{Cor 4.D}
If $\mbox{Im}\,T_g$ is closed, then $g\in \mathcal G L_\infty$.
\end{cor}


\section{Toeplitz operators and paired operators}

It is clear that 
\[
T_g=P^+g\,I_{|H_p^+}= P^+gP^+_{|H_p^+}.
\]
Therefore Toeplitz operators belong to the class of operators of the form
\[
P\,A_{|\mbox{Im}\,P}=P\,A\,P_{\,|\mbox{Im}\,P}
\]
where $P\in \mathcal L(X)$ is a projection in the Banach space $X$ and $A\in \mathcal L(X)$. These operators, which are called {\it operators of Wiener-Hopf (WH) type}, are closely connected with operators in $X$ of the form $PAP+Q$ (where $Q=I-P$) and to the {\it paired operators} $AP+Q$ and $PA+Q$ as follows:

\begin{prop}
\label{Prop 4.1}
Let $X$ be a Banach space, $A\in \mathcal L(X)$, and let $P,Q\in \mathcal L(X)$ be complementary projections, i.e., $P+Q=I$. Then
\[
AP+Q=(PAP+Q)(I+QAP)
\]
\[
PA+Q=(I+PAQ)(PAP+Q)
\]
where $I+QAP$ and $I+PAQ$ are invertible operators with inverses $I-QAP$ and $I-PAQ$, respectively.
\end{prop}

If $X\,,\,Y\,,\,\tilde X\,,\,\tilde Y$ are Banach spaces, we say that two operators $T:X\to \tilde X$ and $S: Y\to \tilde Y$ are {\it equivalent} ($T\sim S$) if and only if $T=ESF$ where $E,F$ are invertible operators. More generally, $T$ and $S$ are \emph{equivalent after extension} if and only if there exist (possibly trivial) Banach spaces $X_0$, $Y_0$, called \emph{extension spaces}, and invertible bounded linear operators $E:\widetilde{Y}\oplus Y_0\rightarrow \widetilde{X}\oplus X_0$ and  $F:X\oplus X_0\rightarrow Y\oplus Y_0$, such that
\beq \label{III.1}
\left(
  \begin{array}{cc}
    T & 0 \\
    0 & I_{X_0} \\
  \end{array}
\right)=E \left(
  \begin{array}{cc}
    S & 0 \\
    0 & I_{Y_0} \\
  \end{array}
\right) F.
\eeq
In this case we say that $T\simast S$ (\cite{BTsk}) .

From Proposition \ref{4.1} we see that 
\beq
\label{4A}
AP+Q\sim PA+Q\sim PAP+Q.
\eeq
On the other hand, taking for instance the operator $AP+Q$,   we can write
\[
\left(
  \begin{array}{cc}
    AP+Q & 0 \\
    0 & I_{|\{0\}} \\
  \end{array}
\right)=E \left(
  \begin{array}{cc}
    P\,A\,P_{\,|\mbox{Im}\,P} & 0 \\
    0 & Q{\,|\mbox{Im}\,Q} \\
  \end{array}
\right) F
\]
where
 \[
F= \left(
  \begin{array}{cc}
    P & -Q \\
    Q & P \\
  \end{array}
\right)\left(
  \begin{array}{cc}
    I+QAP & 0 \\
    0 & I_{|\{0\}} \\
  \end{array}
\right):X\oplus \{0\}\to {\mbox{Im}\,P\oplus \mbox{Im}\,Q}
\]
and 
\[
E= \left(
  \begin{array}{cc}
    P & Q \\
    -Q & P \\
  \end{array}
\right):{\mbox{Im}\,P\oplus \mbox{Im}\,Q}\to {X\oplus \{0\}}
\]
are invertible; thus each one of the three operators in \eqref{4A} is equivalent after extension to $ P\,A\,P_{\,|\mbox{Im}\,P} $. In particular for $P=P^+\,,\,Q=P^-$ we see that
\beq
\label{4B}
T_g=P^+gP^+_{|H_p^+}\sim P^+gI\,+P^-\simast gP^++P^-.
\eeq

Operators which are equivalent after extension share many properties, namely the following (\cite{BTsk}).

\begin{thm}
\label{4.2}
 Let $T: X\rightarrow \widetilde{X}$, $S: Y\rightarrow \widetilde{Y}$ be operators and assume that $T \simast S$. Then
\begin{enumerate}
  \item $\ker T \simeq \ker S$;
  \item $\imag T$ is closed if and only if $\imag S$ is closed and, in that case,  $\widetilde{X}/\imag T \simeq \widetilde{Y}/\imag S$;
  \item if one of the operators $T$, $S$ is generalised (left, right) invertible, then the other is generalised (left, right) invertible too;
  \item $T$ is Fredholm if and only if $S$ is Fredholm and in that case $\dim \ker T=\dim \ker S$, $\codim \imag T=\codim \imag S$.
\end{enumerate}
\end{thm}

It follows from \eqref{4B} and from Theorem \ref{4.2} that we can reduce the study of many properties of Toeplitz operators to the corresponding study for paired operators. We make use of this in the following theorem, which will be used later.

\begin{thm}
\label{Thm 4.3}
 Let $g\in\mathcal GL_\infty$. Then $T_g$ is invertible in $H_p^+$ if and only if $T_{g^{-1}}$ is invertible in $H_{p'}^+$, where $\frac {1}{p}+\frac {1}{p'}=1$.
\end{thm}

\begin{proof}
Taking Theorem \ref{4.2} into account, it is enough to prove that $(gP^++P^-)^*$ is invertible in $H_{p'}^+$ if and only if $g^{-1}P^++P^-$ is invertible in $H_{p'}^+$. Indeed we have, by \eqref{3.5},
\begin{eqnarray}
g^{-1}(P^++gP^-)&&=g^{-1}(CP^-C+gP^-)=g^{-1}C(P^-+\bar gP^+)C
\nonumber\\
&&=g^{-1}C(gP^++P^-)^*C.
\nonumber\\
\end{eqnarray}
\end{proof}


\section{Fredholmness, invertibility and Wiener-Hopf factorisation}

It is easy to see, using properties (P1)-(P3) in Section 3,  that the invertibility properties of $T_{r^k}$ (cf. Proposition \ref{Prop 4.A}) also hold for any Toeplitz operator whose symbol $g$ can be represented as a product

\beq
\label{5.1}
g=g_-r^kg_+\,\,\,,\,\,\mbox{with}\,\,g_\pm\in\mathcal GH_\infty^\pm\,\,,\,\,k\in \mathbb Z.
\eeq
In that case we say that \eqref{5.1} is a {\it bounded WH factorisation} and we have a corresponding factorisation for the Toeplitz operator $T_g$:
\beq
\label{5.2}
T_g=T_{g_-}T_{r^k}T_{g_+}.
\eeq
The integer $k$ is called the {\it index} of the factorisation \eqref{5.1} and the latter is said to be {\it canonical} when $k=0$. Since $T_{g_\pm}$ are invertible, and taking Proposition \ref{Prop 4.A} into account, we have the following.

\begin{thm}
\label{Thm 5.1}
Let $g$ admit a bounded WH factorisation \eqref{5.1}.

(i) If $k=0$ then $T_g$ is invertible with inverse 
\beq
\label{5.3}
(T_g)^{-1}=T_{g_+^{-1}}T_{g_-^{-1}}.
\eeq
(ii) If $k>0$, then $T_g$  is left invertible, $\dim\ker T_g^*=k$ and a left inverse is \beq
\label{5.4}
(T_g)_l^{-1}=T_{r^{-k}}T_{g_+^{-1}}T_{g_-^{-1}}.
\eeq
(iii) If $k<0$, then $T_g$  is right invertible, $\dim\ker T_g=|k|$ and a right inverse is \beq
\label{5.5}
(T_g)_r^{-1}=T_{g_+^{-1}}T_{g_-^{-1}}T_{r^{-k}}.
\eeq
\end{thm}
\begin{proof}
If $g=g_-g_+$  then \eqref{5.3} holds by (P2) in Section 3. If $k>0$, then $r^k\in H_\infty^+\,,\,r^{-k}\in H_\infty^-$  and we have
\[
(T_{r^{-k}}T_{g_+^{-1}}T_{g_-^{-1}})T_g=T_{r^{-k}}T_{g_+^{-1}}T_{g_-^{-1}}T_{g_-}T_{r^{k}}T_{g_+}=T_{r^{-k}g_+^{-1}r^kg_+}=I_+.
\]
We can prove (iii) analogously.
\end{proof}

Left and right inverses are not unique; for $k\neq 0$ another one sided inverse (left or right depending on whether $k>0$ or $k<0$) is given by $T_{g_+^{-1}}T_{r^{-k}}T_{g_-^{-1}}$.

\begin{defn}
\label{defn 5.3}
(\cite{Du, Pross}) Whenever we associate an integer $k$ to an operator $A$, and $A$ is invertible  if $k=0$, only left invertible if $k\in\mathbb Z^+$ and only right invertible when $k\in\mathbb Z^-$, we say that the invertibility of $A$ corresponds to the integer $k$.
\end{defn}

A class of symbols that always admit a bounded WH factorisation is $\mathcal G \mathcal R$, the set of all rational functions that belong to $\mathcal G L_\infty$. It is easy to see that, for $R\in \mathcal G\mathcal R$, we have $R=R_-\,r^kR_+$ where $R_-^{\pm 1}\in \mathcal R\cap H_\infty^-\,,\,R_+^{\pm 1}\in \mathcal R\cap H_\infty^+$ and $k\in\mathbb Z$ is the winding number of $R$ around the origin, $ind\, R$, also called the index of $R$ with respect to the origin (\cite{Ahl}), which is given by the difference between the number of zeroes and the number of poles of $R$ in $\mathbb C^+$.

Thus, for example,
\[
\frac{(x+2i)^2}{x^2+1}=\frac{x+i}{x-i}\left
(\frac{x+2i}{x+i}\right)^2\,\,\,\,\,\,\,(R_-=1\,,\,k=-1\,,\,R_+=\left(\frac{x+2i}{x+i}\right)^2\,).
\]

As a consequence of Theorem \ref{5.1} and Corollary \ref{Cor 4.D} we have thus:

\begin{cor}
\label{cor 5.4}
Let $R\in\mathcal R$. A necessary and sufficient condition for $T_R$ to be Fredholm in $H_p^+$ is that $\inf_{x\in \RR} |R(x)|>0$. In that case the invertibility of $T_R$ corresponds to $k=\mbox{ind}\,R$ and $\dim\ker T_R^*=k$ if $k>0$, $\dim\ker T_R=|k|$ if $k<0$.
\end{cor}
Other classes of continuous functions, called decomposing algebras with the factorisation property (\cite{CG,GK,MP}), such as the Wiener algebra on the real line $W(\mathbb R_\infty)=\mathbb C+\mathcal F L_1$ and the algebra $C_\mu(\mathbb R_\infty)$ of H\H{o}lder continuous functions with exponent $\mu\in]0,1[$, also possess the property that every invertible element of the algebra admits a bounded WH factorisation \eqref{5.1}. As a consequence, the result of Corollary \ref{cor 5.4} also holds if we replace $\mathcal R$ by $W(\mathbb R_\infty)$ or $C_\mu(\mathbb R_\infty)$ and Theorem \ref{5.1} provides formulae for the inverses, or one sided inverses. 

\vspace{5mm}
Using rational approximation, we can show that the results of Corollary \ref{cor 5.4} still hold when we replace $\mathcal R$ by $C(\mathbb R_\infty)$ and we can also obtain formulae for the inverses, or one sided inverses, of the operator $T_g$. However, in contrast to \eqref{5.3}-\eqref{5.5}, these formulae now involve infinite series of operators (\cite{CG,GK}). Extending the previous results to more general classes of symbols in such a way that we can obtain similarly simple explicit formulae requires generalising the concept of factorisation.
This is done in such a way that the main properties of the factorisation which are used to garantee invertibility of $T_g$, when $g$ admits a canonical bounded WH factorisation $g=g_-g_+$, hold at least in a dense subset of $H_p^+$. We choose the latter to be $\mathcal R_0^+$, consisting of the rational functions belonging to $\mathcal R\cap H_p^+$. Thus we look for a factorisation of the form $g=g_-r^kg_+$, where the factors are no longer required to be bounded but the equalities
\[
P^+\,g_+^{\pm1}P^+=g_+^{\pm1}P^+\,,\,P^-\,g_-^{\pm1}P^-=g_-^{\pm1}P^-
\]
still hold in $\mathcal R_0^+$ (cf. \eqref{2.6}). However, unlike the case when $g_+^{\pm1}$ and $g_-^{\pm1}$ are bounded in $\mathbb C^+$ and $\mathbb C^-$ respectively, this is not enough to garantee the invertibility of $T_g$ when $g=g_-g_+$. Having in mind that $(T_{g_-g_+})^{-1}=T_{g_+^{-1}}T_{g_-^{-1}}=g_+^{-1}P^+g_-^{-1}I_+:H_p^+\to H_p^+$  when $g=g_-g_+$ is a bounded factorisation, we will also require now that the operator $g_+^{-1}P^+g_-^{-1}I_+:\mathcal R_0^+\to H_p^+$  is well defined and admits a bounded extension to $H_p^+$.

\begin{defn}
\label{defn 5.5}
A {\it Wiener-Hopf (or generalised) factorisation} of a function $g\in L_\infty$ with respect to $L_p$ (WH $p$-factorisation for short) is a representation $g=g_-r^kg_+$ with $k\in \mathbb Z$ and
\begin{eqnarray}
\label{5.5A}
&g_+\in \mathcal H_{p'}^+\,,\,g_-\in \mathcal H_{p}^-\\
\label{5.5B}
&(g_+)^{-1}\in \mathcal H_{p}^+\,,\,(g_-)^{-1}\in \mathcal H_{p'}^-\\
\label{5.5C}
&g_+^{-1}P^+g_-^{-1}I_+:\mathcal R_0^+=\mathcal R\cap H_p^+\to H_p^+\,\mbox{is a bounded operator}.
\end{eqnarray}
\end{defn}

The integer $k$ is called the {\it index} of the factorisation, and the latter is said to be {\it canonical} when $k=0$. If a WH $p$-factorisation exists, then it is unique up to a constant factor:

\begin{thm}
\label{Thm 5.6}
Let $g \in L_\infty$  and let $g=g_-r^kg_+$ and $g=\tilde g_- r^{\tilde k}\tilde g_+$ be two WH $p$-factorisations. Then $k=\tilde k$ and there is $C\in \mathbb C \setminus \{0\}$ such that $g_-=C\,\tilde g_-\,,\,g_+=C^{-1}\tilde g_+$.
\end{thm}

\begin{proof}
We start by proving the uniqueness of the index. If $k>\tilde k$, then $r^{k-\tilde k}g_+\,\tilde g_+^{-1}=\tilde g_-\,g_-^ {-1}$, which is equivalent to $r^{k-\tilde k-1}\frac{g_+}{x+i}\,\tilde g_+^{-1}=\frac{\tilde g_-}{x-i}\,g_-^ {-1}$. Since $k-\tilde k-1\geq 0$, the left hand side belongs to $\mathcal H_1^+$, the right hand side to $\mathcal H_1^-$; thus both would be zero, which is impossible. We can see analogously that we cannot have $k<\tilde k$. It follows that $g_-\,g_+=\tilde g_-\,\tilde g_+$, which is equivalent to $g_-\,\tilde g_-^{-1}=g_+^{-1}\,\tilde g_+$. Again, since the left hand side belongs to $(x-i)\,\mathcal H_1^-$ and the right hand side to $(x+i)\,\mathcal H_1^+$, we conclude that both sides are equal to a non-zero constant (cf. \eqref{2.12}).
\end{proof}

We now want to prove that $T_g$ is Fredholm in $H_p^+$ if and only if $g$ admits a WH $p$-factorisation, and invertible if and only if this factorisation is canonical. We will need the following lemmas.

\begin{lem}
\label{L5.6}
Let $g_-^{-1}\in \mathcal H_{p'}^-\,,\,r_0\in \mathcal R_0^+$. Then $P^+(g_-^{-1} r_0)\in \mathcal R_0^+$.
\end{lem}

\beginpf
Let $r_0(x)=\frac{1}{(x+z_0)^n}$ with $z_0\in \mathbb C^+\,,\,n\in \mathbb N$. We have
\begin{eqnarray*}
&P^+&(g_-^{-1}R_0)=\\
&P^+&\,
  \frac{g_-^{-1}-\displaystyle\sum_{j=0}^{n-1} (g_-^{-1})_{(-z_0)}^{(j)}(x+z_0)^j}{(x+z_0)^n}+\frac{1}{(x+z_0)^n}\displaystyle\sum_{j=0}^{n-1} (g_-^{-1})_{(-z_0)}^{(j)}(x+z_0)^j\\
&=&\frac{1}{(x+z_0)^n}\displaystyle\sum_{j=0}^{n-1} (g_-^{-1})_{(-z_0)}^{(j)}(x+z_0)^j \in \mathcal R_0^+.
\end{eqnarray*}
\endpf

Note that it follows from Lemma \ref{L5.6} that $g_+^{-1}P^+(g_-^{-1} r_0)\in H_p^+$ $r_0\in \mathcal R_0^+$ and therefore $g_+^{-1}P^+g_-^{-1} I_+:\mathcal R_0^+\to H_p^+$ is well defined.

\begin{lem}
\label{L5.7}
Let $g_-\in \mathcal H_{p}^-\,,\,g_-^{-1}\in \mathcal H_{p'}^-\,,\,r_0\in \mathcal R_0^+$. Then 
\[g_-P^-(g_-^{-1} r_0)\in H_p^-.
\]
\end{lem}

\beginpf
We prove this by induction considering, without loss of generality, that $r_0(x)=\frac{1}{(x+z_0)^n}$ with $z_0\in \mathbb C^+\,,\,n\in \mathbb N$. For $n=1$ we have
\[g_-\,P^-(g_-^{-1} \frac{1}{x+z_0})=g_-\,\frac{g_-^{-1}-g_-^{-1}(-z_0)}{x+z_0}=g_-^{-1}(-z_0)\,\frac{g_-(-z_0)-g_-}{x+z_0}\,\in H_p^-.
\]
Now assume that $g_-P^-(g_-^{-1} \frac{1}{(x+z_0)^n})\in H_p^-$ and let $F_-=P^-(g_-^{-1} \frac{1}{(x+z_0)^n})$. Then
\begin{eqnarray*}
\nonumber
g_-P^- \frac{g_-^{-1}}{(x+z_0)^{n+1}}\,\,\,\;=\;\;\;g_-\,P^-\left[\left(P^-\frac{g_-^{-1} }{(x+z_0)^n}+P^+ \frac{g_-^{-1}}{(x+z_0)^n}\right)\frac{1}{x+z_0}\right]
\end{eqnarray*}
\begin{eqnarray*}
&=&g_-\,P^-[(P^-\frac{g_-^{-1} }{(x+z_0)^n})\frac{1}{x+z_0}]\\
&=&g_-\,P^-\frac{F_-}{(x+z_0)}=g_-\,\frac{F_--F_-(-z_0)}{(x+z_0)})\\
&=&\frac{g_-\,F_--g_-(-z_0)F_-(-z_0)}{(x+z_0)}+F_-(-z_0)\frac{g_-(-z_0)-g_-}{(x+z_0)}\in H_p^-
\end{eqnarray*}
since both terms on the right hand side belong to $H_p^-$,
\endpf

\begin{thm}
\label {thm 5.8}
Let $g\in L_\infty$. The operator $T_g$ is invertible in $H_p^+$ if and only if $g$ admits a canonical WH $p$-factorisation 
\beq
\label{5.6}
g=g_-\,g_+.
\eeq
\end{thm}

\beginpf
(i) First we prove that if $g$ admits a factorisation \eqref{5.6} then $T_g$ is invertible. Assume thus that \eqref{5.6} is a canonical WH $p$-factorisation.

Then $T_g$ is injective because
\[
f_+\in \ker T_g\Leftrightarrow gf_+=f_- \,\,\,\mbox{with}\,\,f_-\in H_p^-
\]
and $gf_+=f_-\Leftrightarrow g_+f_+=g_-^{-1}f_-$\,; since the left hand side of the last equality belongs to $\mathcal H_1^+$ while the right hand side belongs to $\mathcal H_1^-$, both are zero and we conclude that $f_+=0$. On the other hand $T_g$ is surjective. To prove this, let $T_0=g_+^{-1}P^+g_-^{-1} I_+:\mathcal R_0^+\to H_p^+$ and let $r_0\in \mathcal R_0^+$. Then, using Lemmas \ref{L5.6} and \ref{L5.7},
\begin{eqnarray*}
T_gT_0r_0^+&=& P^+gP^+g_+^{-1}P^+(g_-^{-1}r_o^+)=P^+g\,g_+^{-1}P^+(g_-^{-1}r_0^+)\\
&=& P^+g_-P^+(g_-^{-1}r_0^+)=r_0^+-P^+g_-P^-(g_-^{-1}r_0^+)=r_0^+.
\end{eqnarray*}
Now take any $\phi_+\in H_p^+$ and let $(r_n^+)$,  with $r_n^+\in \mathcal R_0^+$ for all $n\in \mathbb N$, be a sequence such that $r_n^+\rightarrow \phi_+$ in $H_p^+$. Let moreover $T$ be the continuous extension of $T_0$ to $H_p^+$. We have $T_gTr_n^+=T_gT_0\,r_n^+=r_n^+$ and it follows that $T_gT\phi_+=\phi_+$, i.e., $\phi_+\in \mbox{Im}\,T_g$.

(ii) Conversely, assume that $T_g$ is invertible in $H_p^+$. By Theorem \ref{Thm 4.3} this is equivalent to $T_{g^{-1}}$ being invertible in $H_{p'}^+$, since we must have $g\in\mathcal GL_\infty$ by Corollary \ref{cor3.5}. Let then $u_+\in H_p^+\,,\,v_+\in H_{p'}^+$ be the unique solutions of
\[
T_g\,u_+=\lambda_+^{-1}\,\,,\,\,T_{g^{-1}}\,v_+=\lambda_+^{-1}
\]
where $\lambda_+$ is defined in \eqref{2.8}. Then we have
\beq
\label{5.7}
gu_+=\lambda_+^{-1}\,+u_-\,\,,\,\,g^{-1}v_+=\lambda_+^{-1}\,+v_-\,\,\,\mbox{with}\,\,u_-\in H_p^-\,,\,v_-\in H_{p'}^-\,.
\eeq
Multiplying these two equations we get
\[
\lambda_+\,u_+\,v_+\,-\lambda_+^{-1}=u_-\,+v_-\,+\,\lambda_+\,u_-\,v_-.
\]
Since the left hand side is in $\mathcal H_1^+$ and the right hand side is in $\mathcal H_1^-$ we conclude that both are zero, therefore
\beq
\label{5.8}
\lambda_+\,u_+\,v_+=\,\lambda_+^{-1}\,\,\,,\,\,\,u_-\,+v_-\,\,+\lambda_+\,u_-\,v_-=0.
\eeq
The first equality in \eqref{5.8} implies that $(\lambda_+\,u_+)(\lambda_+\,v_+)=1$; thus, defining $g_+=\lambda_+\,v_+$, we have that $g_+\in\mathcal H_{p'}^+\,,\,g_+^{-1}=\lambda_+\,u_+\in \mathcal H_{p}^+$. From the second equality in \eqref{5.8} we have $(1+\lambda_+\,v_-)(1+\lambda_+\,u_-)=1$. So, defining $g_-=1+\lambda_+\,u_-=1+\lambda_-\,r^{-1}\,u_-$, we have $g_-\in \mathcal H_{p}^-\,,\,g_-^{-1}=1+\lambda_+\,v_-\in\mathcal H_{p'}^-$. From \eqref{5.7}, $g=g_-\,g_+$, so it is left to show that $g_+^{-1}P^+g_-^{-1} I_+:\mathcal R_0^+\to H_p^+$ is bounded.

Let $r_o\in \mathcal R_0^+$ and let $f_+=(T_g)^{-1}r_0$. Then $P^+(gf_+)=r_0$, i.e., $gf_+=r_0\,+P^-(gf_+)$. Now
\begin{eqnarray*}
\nonumber
gf_+=r_0+P^-(gf_+)&\Leftrightarrow& g_-^{-1} gf_+=g_-^{-1}   r_0\,+  g_-^{-1} P^-(gf_+)\\
&\Leftrightarrow& g_+f_+\,-\,P^+(g_-^{-1}r_0)=P^-(g_-^{-1}r_0)+g_-^{-1}P^-(gf_+).
\end{eqnarray*}
The left hand side of the last equality belongs to $\mathcal H_1^+$, while the right hand side belongs to $\mathcal H_1^-$, so both are equal to zero and we conclude that
\beq
\label{5.9}
g_+f_+=P^+(g_-^{-1}r_0)\in H_{p'}^+\,\,,\,\,g_-^{-1}P^-(gf_+)=-P^-(g_-^{-1}r_0)\in H_{p'}^-.
\eeq
Taking these relations into account, we get 
\begin{eqnarray*}
\nonumber
g_+^{-1}P^+(g_-^{-1}r_0)&=&g_+^{-1}P^+(g_-^{-1}T_gf_+)=g_+^{-1}P^+(g_-^{-1}P^+(gf_+))\\
&=&g_+^{-1}P^+\left [g_-^{-1}\left (gf_+\,-P^-(gf_+)\right )\right)]\\
&=&g_+^{-1}P^+[\,\underbrace{g_+f_+}_{\in H_{p'}^+}\,-\underbrace{g_-^{-1}P^-(gf_+)}_{\in H_{p'}^-}]=f_+=(T_g)^{-1}r_0
\end{eqnarray*}
and it follows that $g_+^{-1}P^+g_-^{-1} I_+:\mathcal R_0^+\to H_p^+$ is bounded.
\endpf

\begin{rem}
\label{rem 5.8A}
Remark that \eqref{5.5C} was not needed to prove the injectivity of $T_g$ and the relation $\mathcal R_0^+\subset Im\,T_g$ in the first part of the proof above.
\end{rem}

\begin{thm}
\label {5.9}
Let $g\in L_\infty$. The operator $T_g$ is Fredholm in $H_p^+$ if and only if $g$ admits a WH $p$-factorisation $g=g_-r^kg_+$ and, in that case, its Fredholm index is $\mbox{Ind}\,T_g=-k$.
\end{thm}

\beginpf
(i) Assume that $T_g$ is Fredholm and let $-k$ be its Fredholm index. Then if $k\leq 0$ we have $r^{-k}\in H_\infty^+$ and $\mbox{Ind}\,T_{r^{-k}}=k$. Therefore $T_gT_{r^{-k}}=T_{gr^{-k}}$ is a Toeplitz operator which is Fredholm with index zero. By Corollary \ref{Cor 4.D}, $T_{gr^{-k}}$ is invertible. Therefore $gr^{-k}$ admits a canonical WH $p$-factorisation $gr^{-k}=g_-g_+$ and it follows that $g=g_-r^kg_+$ is a WH $p$-factorisation. If $k\geq 0$ we have $r^{-k}\in H_\infty^-$ and $\mbox{Ind}\,T_{r^{-k}}=k$, so $T_{r^{-k}}T_g=T_{gr^{-k}}$ is a Toeplitz operator which is Fredholm with index zero, therefore invertible, and we conclude analogously that $g$ admits a canonical WH $p$-factorisation.

(ii) Conversely, let us now assume that $g=g_-r^kg_+$ is a WH $p$-factorisation. If $k\geq 0$ then $T_g=T_{g_-g_+}T_{r^k}$; if $k\leq 0$ then $T_g=T_{r^k}T_{g_-g_+}$. Since $T_{g_-g_+}$ is invertible by Theorem \ref{thm 5.8} and $T_{r^k}$ is Fredholm with index $-k$, we conclude that $T_g$ is Fredholm with index $-k$.
\endpf

\begin{cor}
\label {Cor 5.10}
If $g\in L_\infty$ admits a WH $p$-factorisation $g=g_-r^kg_+$, then the invertibility of $T_g$ corresponds to the integer $k$. For $k>0$ the operator $g_+^{-1}P^+g_-^{-1}T_{r^{-k}}$ is a left inverse for $T_g$ and $\dim\ker T_g^*=k$; for $k<0$ the operator  $T_{r^{-k}}(g_+^{-1}P^+g_-^{-1}I_+)$  is a right inverse for $T_g$ and $\dim\ker T_g=|k|$.                    
\end{cor}

It is clear from Definition \ref{defn 5.5} that $g\in L_\infty$ admits a WH $p$-factorisation $g=g_-r^kg_+$ if and only if $g_0=r^{-k}g$ admits a canonical WH $p$-factorisation. In that case $T_{g_0}$ is invertible by Theorem \ref{thm 5.8}. Therefore the factors $g_\pm$ satisfy \eqref{5.5A}, \eqref{5.5B} and the range of $T_{g_0}=T_{g_-g_+}$ is closed in $H_p^+$, i.e.,
\beq
\label{5.5D}
T_{r^{-k}g}=T_{g_-g_+} \,\,\mbox{has closed range}.
\eeq
We can see from the proof of Theorem \ref{thm 5.8} that the converse is also true and thus condition \eqref{5.5C} in Definition \ref{5.5} can be replaced by \eqref{5.5D}. In fact, from the first part of the proof of Theorem \ref{thm 5.8}  we see that \eqref{5.5A}
 and \eqref{5.5B} imply that $T_{g_0}$ is injective and $\mathcal R_0^+\subset \mbox{Im}\,T_{g_0}$. Since $\mathcal R_0^+$ is dense in $H_p^+$, by adding \eqref{5.5D} to  \eqref{5.5A} and \eqref{5.5B} we conclude that $T_{g_0}$ is surjective, therefore invertible. Thus we have the following.

 \begin{cor}
\label {Cor 5.11}
The factorisation $g=g_-r^kg_+$, with $k\in \mathbb Z$, is a WH $p$-factorisation for $g\in L_\infty$ if and only if $g_\pm$ satisfy \eqref{5.5A}, \eqref{5.5B} and \eqref{5.5D}.                  
\end{cor}

This means that we can replace \eqref{5.5C} by another condition which is easy to verify for a large class of functions in $L_\infty$, including all piecewise continuous functions (\cite{Du, GK}), as we show in section 6.

\vspace{5mm}
We can now ask when does a function $g\in L_\infty$ admit a WH $p$-factorisation and, if it exists, how to obtain its factors and, in particular, its index.We address this question in the following section.


\section{Piecewise continuous symbols}

There are several classes of functions for which simple criteria for existence of a WH $p$-factorization can be established. (see for instance \cite{BSilb, CG, MP}). We will consider here only the class $PC(\mathbb R_\infty)$ of all piecewise continuous functions with finite limits at $\pm \infty$.

\vspace{5mm}
To each $g\in PC(\mathbb R_\infty)$ we associate 
\beq
\label{6.1}
g_p(x,w) =\frac{g(x^-)+g(x^+)}{2}+\frac{g(x^-)-g(x^+)}{2}\, coth\,\left ((\pi(\frac{i}{p}+w)\right)
\eeq
for $x\in \mathbb R_\infty\,,\,w\in \mathbb R_\pm\infty$, where $\mathbb R_\infty$ and $\mathbb R_\pm\infty$ are the one-point and the two-point compactifications of the real line, respectively, and for $x=\infty$ we  take $g(\infty^-)=g(+\infty)\,,\,g(\infty^+)=g(-\infty)$ (\cite{Du}).

If $g\in C(\mathbb R_\infty)$ then $g_p(x,w)=g(x)$ for all $x\in \mathbb R_\infty$ and $w\in \mathbb R_{\pm\infty}$. It is also easy to see that if $p=2$ then 
\beq
\label{6.2}
g_2\,(x,w) =\frac{g(x^+)+g(x^-)}{2}+\frac{g(x^+)-g(x^-)}{2}\,th(\pi w).
\eeq
The image of $g_p$ in the complex plane is a closed curve $\Gamma$ obtained by adding to the image of $g(x)$, with $x\in \mathbb R_\infty$ such that $g$ is continuous at $x$, certain arcs of a circle (or line segments, when $p=2$) connecting the points corresponding to $g(x^-)$ and $g(x^+)$ whenever these two values are different. If 
\beq
\label{6.3}
\inf\limits_{(x,w)\in \,\mathbb R_\infty\times\,\mathbb R_\pm\infty}|g_p(x,w)|>0
\eeq
then we can associate with $g_p$ an integer, designated by $ind\,g_p$, defined as the winding number of $\Gamma$ around the origin. If $g\in C(\mathbb R_\infty)$ then condition \eqref{6.3} means that $g\in \mathcal G C(\mathbb R_\infty)$ and in that case $ind\,g_p=ind\,g$.

\begin{defn}
\label {6.1}
We say that $g\in PC(\mathbb R_\infty)$ is $p$-nonsingular, or $p$-regular, if and only if \eqref{6.3} holds.
\end{defn}

Note that the product of two $p$-regular functions may be $p$-singular, as shown in an example presented below in this section. However, if two functions $g\,,\,h\in PC(\mathbb R_\infty)$ have no common points of discontinuity, then $(gh)_p(x,w)=g_p(x,w)\,h_p(x,w)$ and, if $gh$ is $p$-regular, then $ind\,(gh)_p=ind\,g_p\,ind\,h_p$ (\cite{Du}).

\vspace{5mm}
For Toeplitz operators with piecewise continuous symbols we have the following.

\begin{thm}
\label {6.2}
(\cite{Du, W}) Let $g\in PC(\mathbb R_\infty)$. The operator $T_g$ has closed range in $H_p^+$ if and only if $g$ is $p$-regular. In this case $T_g$ is Fredholm, the invertibility of $T_g$ corresponds to $k=ind\,g_p$ and the Fredholm index of $T_g$ is $\mbox{Ind}\,T_g=-k$.
\end{thm}

\begin{cor}
\label {Cor 6.3}
If $g\in PC(\mathbb R_\infty)$, then $g$ admits a WH $p$-factorisation if and only if $g$ is $p$-regular and, in this case, $g=g_-r^kg_+$ with $k=ind\,g_p$.
\end{cor}

In the case $p=2$ we see from the previous corollary that $g\in PC(\mathbb R_\infty)$ is $2$-regular if and only if $g$ is locally sectorial, i.e., for all $x\in \mathbb R_\infty$ we have 
\beq
\label{6.4}
g(x^\pm)\neq 0
\eeq
\beq
\label{6.5}
\frac{g(x^+)}{g(x^-)}= e^{2\pi i \alpha}\,\,\,\,\mbox{with}\,\,\,|\Re (\alpha)|<1/2.
\eeq
Condition \eqref{6.5} means precisely that the discontinuities are such that the line segment connecting the points $g(x^-)$ and $g(x^+)$ in the complex plane does not include the origin.

\vspace{5mm}
Now, having considered the Fredholm properties of $T_{r^k}$ with $k\in\mathbb Z$, in Proposition \ref{4A}, it is natural to consider also the case when the exponent $k$ is not an integer, to illustrate the previous results.

Let $p=2\,,\,\alpha \in \mathbb R\setminus\mathbb Z\,\,\,\mbox{with}\,\,\,|\alpha|\leq 1/2$, and let $g_{(\alpha,\infty)}(x)=\left(\frac{x-i}{x+i}\right)^\alpha_\infty = e^{\alpha\,log\frac{x-i}{x+i}}$ where $log\,w=log\,|w|+i\,arg(w)$ with $arg(w)\in[0\,,\,2\pi[$. In this case $g_{(\alpha,\infty)}(x)$ is analytic in the whole complex plane except for the branch cuts $\pm i\,[0,\infty[$ in the imaginary axis; thus $\left(\frac{x-i}{x+i}\right)^\alpha_\infty$ is continuous on $\mathbb R$, but it has a discontinuity at $\infty$: since $\lim_{x\to +\infty}\frac{x-i}{x+i}=2\pi\,,\,\lim_{x\to -\infty}\frac{x-i}{x+i}=0$, we have
\[
\left(\frac{x-i}{x+i}\right)^\alpha_\infty\,(-\infty)=e^{i2\pi \alpha}\,\,,\,\,\,\left(\frac{x-i}{x+i}\right)^\alpha_\infty\,(+\infty)=1
\]
The image of $(g_{(\alpha,\infty)})_2(x,w)$ in the complex plane consists of the upper half of the unit circle and a line segment connecting the points $1$ and $-1$, when $\alpha=1/2$; it consists of an arc of the unit circle in the upper half-plane connecting the points $e^{2i\pi\alpha}$ and $1$, and a line segment from $1$ to $e^{2i\pi\alpha}$, when $0<\alpha< 1/2$.
It follows from Corollary \ref{Cor 6.3} that $g_{(\alpha,\infty)}$ admits a WH $2$-factorisation  if and only if $|\alpha|<1/2$ and, in that case, the factorisation is canonical.

\vspace{5mm}
We see in particular that $g_{(1/4,\infty)}$ is $2$-regular but $g_{(1/4,\infty)}^2=g_{(1/2,\infty)}$ is $2$-singular, which illustrates our previous assertion that the product of two $2$-regular functions may be $2$-singular if they have common discontinuities.

\vspace{5mm}
More generally, defining for $c\in \mathbb R\,,\,\alpha\in \mathbb C\setminus\mathbb Z$,
\[
(z)^\alpha_c=e^{\alpha\,log\,z}
\]
where the branch cut connecting $0$ to $\infty$ intersects the unit circle at the point $z_0=\frac{c-i}{c+i}$, the function $\left(\frac{x-i}{x+i}\right)^\alpha_c$ is continuous for all points of $\mathbb R_\infty$ except the point $c$. If the discontinuity points of $g\in PC(\mathbb R_\infty)$ are $c_1, c_2,...c_n \in \mathbb R$ and (eventually) $\infty$, then $g$ can be represented in the form (\cite{Du})
\beq
\label{6.6}
g(x)=g_0(x) \left(\frac{x-i}{x+i}\right)^{\alpha_0}_{\infty} \,\,\prod_{j=1}^n \left(\frac{x-i}{x+i}\right)^{\alpha_j}_{c_j}
\eeq
where $g_0\in C(\mathbb R_\infty)$ and
\[
\alpha_j=\frac{1}{2\pi i}\,log\,\frac{g(c_j^-)}{g(c_j^+)}\,\,,\,\,\mbox{for}\,\,j=1, 2, ..., n\,,\,\,\,\,\,\alpha_0=\frac{1}{2\pi i}\,log\,\frac{g(+\infty)}{g(-\infty)}.
\]

Using Corollary \ref{6.3} we have the following:

\begin{thm}
\label {thm 6.4}
(\cite{Du, W}) Let $g\in PC(\mathbb R_\infty)$ be given by \eqref{6.6}. The operator $T_g$ is Fredholm if and only if $-1/p<\Re (\alpha_0)<1-1/p\,,\,1/p-1<\Re (\alpha_j)<1/p$ for all $j=1, 2, ..., n$; in that case its Fredholm index is $k=-ind\,g_0$.
\end{thm}

\vspace{5mm}
Although there are explicit formulae for the factors in a WH $p$-factorisation of $g\in PC(\mathbb R_\infty)$, once its index is known (\cite{W}), the factors can in some cases be obtained by inspection. For example, if $-1/2<\alpha<1/2$ we can write
\beq
\label{6.7}
\left(\frac{x-i}{x+i}\right)^\alpha_\infty=(x-i)^\alpha\,\left(\frac{1}{x+i}\right)^\alpha
\eeq
choosing appropriate branches such that $(x-i)^{\pm\alpha}\in \mathcal H_2^-\,,\,\left(\frac{1}{x+i}\right)^{\pm\alpha}\in \mathcal H_2^+$. Since the range of  $T_g$, with $g$ given by the left hand side of \eqref{6.7}, is closed by Theorem \ref{6.2}, it follows from Corollary \ref{Cor 5.11} that \eqref{6.7} is a canonical WH $2$-factorisation of $g$.


\section{The spectrum of the singular integral operator $S_J$ in $L_2(J)$}

Let $S_J$, where $J\subset \mathbb R$ is an interval, denote the singular integral operator
\beq
\label{7.1}
S_J:L_2(J)\to L_2(J)\,\,\,,\,\,\,S_J\,\,\phi(x)=\frac{1}{\pi i} \int_J{\frac{\phi(t)}{t-x}\,dt}\,\,,\,\,x\in J.
\eeq
In this section we study the spectrum of $S_J$ in three cases, $J=\mathbb R\,,\,J=\mathbb R^+\,,\,J=[0,1]$, using the results of the previous sections. The spectrum of $S_J$ in these three cases was described in \cite{W} using a slightly different approach. Furthermore, using the factorisation of the symbol of an associated Toeplitz operator, we obtain expressions for the resolvent operator $(S_{\mathbb R^+}-\lambda I_{\mathbb R^+})^{-1}$ when $\lambda \notin \sigma(S_{\mathbb R^+})$. Here $I_{\mathbb R^+}$ denotes the identity operator in $L_2(\mathbb R^+)$.

\vspace{5mm}
\emph{First case:} $J=\mathbb R$

Using \eqref{3.4} we can write $S_\mathbb R - \lambda I$ as a paired operator:
\beq
\nonumber
S_\mathbb R - \lambda I=P^+-P^--\lambda(P^++P^-)=(1-\lambda)P^+-(1+\lambda)P^-.
\eeq
Clearly $\lambda=\pm 1$ belong to the point specrum of $S_\mathbb R$, since $S_\mathbb R\,-\,I=-2P^-$, $S_\mathbb R+I=2P^+$ and $P^\pm$ have infinite dimensional kernels. For $\lambda=\pm 1$ the operator $S_\mathbb R - \lambda I$ is invertible, with inverse $\frac{1}{1-\lambda^2}(S_\mathbb R + \lambda I)$. Thus $\sigma(S_\mathbb R)=\{1,-1\}$.

\vspace{5mm}
\emph{Second case:} $J=\mathbb R^+$.

For $\phi\in L_2(\mathbb R^+)$ we have
\beq
\nonumber
S_{\mathbb R^+}\,\,\phi(x)=\frac{1}{\pi i} \int_{\mathbb R^+}{\frac{\phi(t)}{t-x}\,dt}\,\,,\,\,x\in \mathbb R^+.
\eeq
Let $\chi_\pm$ denote the characteristic functions of $\mathbb R^\pm$, respectively. Identifying $L_2(\mathbb R^+)$ with a subspace of $L_2(\mathbb R)$, $L_2(\mathbb R^+)=\chi_+L_2(\mathbb R)$, we have
\beq
\label{7.2}
S_{\mathbb R^+}-\lambda I_{R^+}=\chi_+\,(S_{\mathbb R}-\lambda I)_{|{\chi_+L_2}}.
\eeq
We use now the Fourier transform, a natural tool to deal with convolution integrals,
\beq
\label{7.3}
\mathcal F\,f(x)=\frac{1}{\surd\overline{2\pi}}\int_\mathbb R{f(x)\, e^{ixt}dt}=\frac{1}{\surd\overline{2\pi}}\lim_{A\to\infty}\int_{-A}^A{f(x)\, e^{ixt}dt}
\eeq
which is an isometric isomorphism from $\chi_+\,L_2$ onto $H_2^+$. We have
\beq
\label{7.4}
P^+=\mathcal F\,\chi_+\,\mathcal F^{-1}=\mathcal F^{-1}\,\chi_-\,\mathcal F\,\,\,,\,\,\,P^-=\mathcal F\,\chi_-\,\mathcal F^{-1}=\mathcal F^{-1}\,\chi_+\,\mathcal F\,.
\eeq
Thus we can reduce the study of the invertibility of $S_{\mathbb R^+}-\lambda I_{R^+}$ to the study of the invertibility of the Toeplitz operator
\[
T_{-(sign+\lambda)}:H_2^+\to H_2^+
\]
where $sign\,(x)=\chi_+(x)-\chi_-(x)$, since, for all $f_+\in H_2^+$.
\begin{eqnarray*}
\nonumber
T_{(sign+\lambda)}f_+&=&P^+(sign\,x+\lambda)f_+=\mathcal F\,\chi_+\,\mathcal F^{-1}(\chi_+-\chi_-+\lambda)\mathcal F F_+\\
&=& \mathcal F\,\chi_+\,(\mathcal F^{-1}\chi_+\mathcal F-\mathcal F^{-1}\,\chi_-\,\mathcal F\,+\lambda)F_+\\
&=&\mathcal F\chi_+(P^--P^++\lambda)\,F_+=-\mathcal F\chi_+(S_\mathbb R-\lambda I)F_+\\
&=&-\mathcal F\chi_+(S_\mathbb R-\lambda I)\mathcal F^{-1}f_+=-\mathcal F(S_{\mathbb R^+}-\lambda I_{R^+})\mathcal F^{-1}f_+\,.
\end{eqnarray*}
The symbol $g_\lambda(x)=-sign\,(x)-\lambda$ is $2$-regular if and only if $\lambda\notin [-1,1]$; in this case, $ind\,(g_\lambda)_{2}=0$ and therefore, by Theorem \ref{6.2}, the operator $T_{g_\lambda}$ is invertible. For $\lambda \in [-1,1]$, the range of $T_{g_\lambda}$ is not closed. We conclude therefore that $\sigma (S_{\mathbb R^+})=[-1,1]$.
Since $g_\lambda(x)$ is real for all $x$, $T_{g_\lambda}=T_{g_\lambda}^*$ and by Coburn's Lemma $\ker T_{g_\lambda}=\ker T_{g_\lambda}^*=\{0\}$. 

\vspace{5mm}
\emph{Third case:} $J=[0,1]$.
Let $Y:L_2([0,1])\to \chi_+L_2$ be the isometric isomorphism
\beq
\label{7.5}
Y\,\phi(t)=e^{-t/2}\phi(e^{-t})\chi_+(t).
\eeq
Identifying $L_2(\mathbb R^+)$ with $\chi_+ L_2$ as usual, we define
\beq
\nonumber
M=YJY^{-1}:L_2(\mathbb R^+)\to L_2(\mathbb R^+)\,\,,\,
\eeq
\beq
\label{7.6}
Mf(x)=\frac{1}{\pi i }\int_{\mathbb R^+}{\frac{e^{1/2(t-x)}}{1-e^{t-x}}}\,f(t) dt=:s(t)\,\,\,,\,\,\,t>0.
\eeq
$M$ is a convolution operator: $Mf=\chi_+E\ast f$ where 
\beq
\label{7.7}
E(x)=\frac{e^{-\frac{x}{2}}}{1-e^{-x}}\,\,\,,\,\,\,x\in \mathbb R.
\eeq
Applying the Fourier transform to $M$, and taking into account that (\cite{Du})
\beq
\label{7.8}
\left(\mathcal F E\right)(w)=th(\pi w)\,\,\,,\,\,\,w\in \mathbb R,
\eeq
we have, for all $f_+\in H_2^+$,
\beq
\label{7.9}
\mathcal F M\mathcal F^{-1}f_+=\mathcal F \left(\chi_+E\ast (\mathcal F^{-1}f_+)\right)=P^+th(\pi x)f_+=T_{th(\pi x)}f_+.
\eeq
Thus the spectrum of $S_J$ is the spectrum of the Toeplitz operator with symbol $th(\pi x)$ and, since the range of this symbol is $[-1,1]$, we conclude as in the previous case that $\sigma (S_J)=[-1,1]$.

\vspace{5mm}
Let us now apply the previous results and a WH $p$-factorisation of the symbol $g_\lambda(x)=-(sign\,x+\lambda)$ to obtain an expression for the resolvent $(S_{\mathbb R^+}-\lambda\,I_{\mathbb R^+})^{-1}$ with $\lambda\in\mathbb C\setminus[-1,1]$. Defining, for $z\neq 0$, $\log z=\log |z|+i\arg z$ with $\arg z \in ]-\pi,\pi]$, let $z^\alpha=e^{\alpha \log z}$ for 
\beq
\label{7.10}
\alpha =\frac{1}{2\pi i}\log\frac{\lambda -1}{\lambda+1}\,.
\eeq
Note that $|\Re (\alpha)|<1/2$ and
\beq 
\nonumber
(-ix)^\alpha=\exp (\alpha \log |x|-i\alpha\frac{\pi}{2}\,sign\,x) \in \mathcal H_2^+
\eeq
\beq 
\nonumber
(ix)^{-\alpha}=\exp (-\alpha \log |x|-i\alpha\frac{\pi}{2}\,sign\,x)  \in \mathcal H_2^-\,,
\eeq
so that $(-ix)^\alpha\,(ix)^{-\alpha}=e^{-i\alpha\pi\,sign \,x}$.
Now, since $\frac{\lambda -1}{\lambda+1}=\exp(i2\pi \alpha)$, we have
\begin{eqnarray*}
\nonumber
\lambda+sign\,x&=&\lambda+1=(\lambda+1)\exp (i\alpha\pi)(-ix)^\alpha(ix)^{-\alpha}\,\,\,,\,\,\,\mbox{for}\,\,x>0,\\
\lambda+sign\,x&=&(\lambda+1)\exp (i2 \alpha\pi)=(\lambda+1)\exp (i \alpha\pi)(-ix)^\alpha(ix)^{-\alpha}\,\,\,,\,\,\,\mbox{for}\,\,x<0.
\end{eqnarray*}
Defining
\begin{eqnarray*}
\nonumber
(g_\lambda)_+&=&(\lambda+1)\exp (i\alpha\pi)(-ix)^\alpha\\
(g_\lambda)_-&=&(ix)^{-\alpha}\,,
\end{eqnarray*}
we conclude that $g_\lambda=-(g_\lambda)_-\,(g_\lambda)_+$ is a canonical WH $2$-factorisation (cf. Corollary \ref{Cor 5.11}) and
\beq
\nonumber
\left(T_{-(sign\,x+\lambda)}\right)^{-1}=-(g_\lambda)_+^{-1}P^+(g_\lambda)_-^{-1}I_{|H_2^+}=(g_\lambda)_+^{-1}\mathcal F\,\chi_+\,\mathcal F^{-1}(g_\lambda)_-^{-1}I_{|H_2^+}.
\eeq
Let us now consider the operator
\beq
\nonumber
(S_{\mathbb R^+}-\lambda I_{\mathbb R^+})^{-1}=\mathcal F^{-1}(T_{-(sign\,x+\lambda)})^{-1}\mathcal F_{|(\chi_+L_2)}.
\eeq
Assume that $-1/2<\Re(\alpha)<0$ (the case where $0<\Re(\alpha)<1/2$ can be studied analogously).
Since (\cite{GrRy})
\beq
\nonumber
\left(\mathcal F^{-1}(g_\lambda)_-^{-1}\right)(t)=\chi_-(t)\,\frac{(-t)^{-\alpha -1}}{\Gamma(-\alpha)}=:K_-(t)
\eeq
and
\beq
\nonumber
(g_\lambda)_+^{-1}(x)=\frac{e^{-i\alpha \pi}}{\lambda+1}(-ix)(-ix)^{-\alpha -1}
\eeq
with
\beq
\nonumber
\left(\mathcal F^{-1}(-ix)^{-\alpha -1}\right)(t)=\frac{\chi_+(t)\,t^\alpha}{\Gamma(1+\alpha)}=:K_+(t)\,,
\eeq
we have, for $\phi_+\in L_2(\mathbb R^+)$,
\begin{eqnarray}
\nonumber
\left(\mathcal F^{-1}(g_\lambda)_+^{-1}\mathcal F (\chi_+\mathcal F^{-1}(g_\lambda)_-^{-1}\mathcal F\phi_+)\right)(t)=\,\,\,\,\,\,\,\,\,\,\,\,\,\;\;\;\;\;\;\;\;\;\;\,\\
=\frac{e^{-i\alpha \pi}}{(\lambda+1)\Gamma(-\alpha)\Gamma(1+\alpha)} \,\chi_+\,\frac{d}{dt}\int_0^t{(t-u)^\alpha\int_u^\infty{(s-u)^{-\alpha-1}f(s)ds}du}\,.\nonumber
\end{eqnarray}
Taking into account that 
\beq
\nonumber
\Gamma(-\alpha)\Gamma(1+\alpha)=\frac{\pi}{\sin (\alpha\pi)}=\frac{-2i\pi}{e^{i\alpha \pi}-e^{-i\alpha \pi}}
\eeq
we have, for $t>0$,
\begin{eqnarray}
(S_{\mathbb R^+}-\lambda I_{\mathbb R^+})^{-1}\,f(t)=\frac{1}{ i \pi} \frac{1}{(\lambda^2 -1)} \, \frac{d}{dt} \left[ \int_0^t (t-x)^\alpha   \left(   \int_x^{\infty}
\frac{f(s)}{(s-x)^{\alpha+1}} ds \right) dx  \right]\,.\nonumber
\end{eqnarray}
Since
\begin{eqnarray}
&&\frac{d}{dt} \left[ \int_0^t (t-x)^\alpha   \left(   \int_x^{\infty}
\frac{f(s)}{(s-x)^{\alpha+1}} ds \right) dx  \right]  \nonumber\\
&&=
\frac{d}{dt} \left[ \int_0^t f(s)   \left(   \int_0^s  
\frac{(t-x)^\alpha }{(s-x)^{\alpha+1}} dx \right) ds  
+ \int_t^{\infty} f(s)   \left(   \int_0^t 
\frac{(t-x)^\alpha }{(s-x)^{\alpha+1}} dx \right) ds 
\right] \nonumber\\
&&=
\frac{d}{dt} \left[ \int_0^t f(s)   \left(   \int_0^{s/t}  
\frac{1 }{(1-y) \, y^{\alpha+1}} dy \right) ds  
- \int_t^{\infty} f(s)   \left(   \int_{s/t}^{\infty} 
\frac{1 }{(1-y) \, y^{\alpha+1}} dy \right) ds  
\right] \nonumber\\
&&= f(t) \int_0^{\infty} \frac{1 }{(1-y) \, y^{\alpha+1}} dy
- t^\alpha \int_0^{\infty} \frac{f(s)}{(t-s) \, s^\alpha} ds \nonumber\\
&&= i \pi \lambda \,  f(t) -  t^\alpha \int_0^{\infty} \frac{f(s)}{(t-s) \, s^\alpha} ds \;, \nonumber
\end{eqnarray}
where we employed the change of variables
\begin{equation}
x = t - \frac{(t-s)}{1-y} \;,
\nonumber
\end{equation}
and the relation
\begin{equation}
\int_{\mathbb{R}^+}  \frac{y^{-\alpha-1} }{ \pi (1-y) } dy = -{\rm cotg} (\alpha \pi) = i \lambda 
\nonumber
\end{equation}

where the integral is understood as a principal value. 

Hence, 
\begin{equation}
(S_{\mathbb R^+}-\lambda I_{\mathbb R^+})^{-1}f(t) = \frac{1}{(\lambda^2 -1 )} \left[ \lambda \, f(t) - \frac{t^\alpha}{i \pi}
\int_{\mathbb{R}^+}  \frac{f(s)}{(t-s) s^\alpha} ds \right] \;\;\;\;\;\; (t > 0) \;.
\nonumber
\end{equation}


\begin{thebibliography}{10}





\bibitem{Ahl}
Ahlfors, Lars V.,
{\em Complex analysis}, 2nd ed.
 McGraw’Hill, 1966.


\bibitem{BTsk}
H. Bart and V.\`E. Tsekanovski\u\i,   Matricial coupling and equivalence after extension. {\em Operator theory and complex analysis (Sapporo, 1991)}, 143--160, Oper. Theory Adv. Appl., 59, Birkh\"auser, Basel, 1992. 
 

\bibitem{BSilb}
B\"ottcher, A. and Silberman, B., 
{\em Analysis of Toeplitz Operators}.
Springer-Verlag, Berlin, 1990.

\bibitem{CDR}
M.C. C\^amara, C. Diogo and L. Rodman, 
Fredholmness of Toeplitz operators and corona problems. 
{\em J. Funct. Anal.} 259 (2010), no. 5, 1273--1299.

\bibitem{CG}
 Clancey, K., Gohberg, I.: \emph{Factorization of Matrix
Functions and Singular Integral Operators}. Birkh\"{a}user, 1981.


\bibitem{Du}
Duduchava, R.,
{\em Integral equations in convolution with discontinuous presymbols. Singular integral equations
with fixed singularities, and their applications to some problems of mechanics}.
Teubner-Texte zur Mathematik, Leipzig, 1979.

\bibitem{duren}
P. L. Duren, {\em Theory of $H^p$ spaces}. Dover, New York, 2000.


\bibitem{GF}
Gohberg, I. and Feldman, I., {\em Convolution equations and projection methods for their solutions}. Transl. Math. Monographs, vol. 41, Amer. Math. Soc., 1974.


\bibitem{GGK}
Gohberg, I., Goldberg, S. and Kaashoek, M., {\em  Basic Classes of Linear Operators}. Birkh\"{a}user , Basel, Boston, Berlin
2003.

\bibitem{GK}
Gohberg, I. and Krupnik, N., {\em One-dimensional linear singular integral equations}, Vols. I and II. Birkhäuser Verlag, Basel, Boston, Berlin 1992.

\bibitem{GrRy}
Gradshteĭn, I. and  Ryzhik, I.M.,  {\em Table of Integrals, Series, and Products}.
Academic Press, 1994

\bibitem{LS}
 Litvinchuk, G. S.; Spitkovskii, I. M., \emph{Factorization of measurable matrix functions}. Operator Theory:
Advances and Applications, 25. Birkh$\ddot{a}$user, Basel, 1987.


\bibitem{MP}
S.G. Mikhlin and S.\  Pr\"ossdorf, {\em Singular integral operators}. Translated from the German by Albrecht B\"ottcher and Reinhard Lehmann. Springer-Verlag, Berlin, 1986.


\bibitem{O}
Okikiolu, G. O., 
{\em Aspects of the Theory of Bounded Integral Operators in Lp-spaces}.
Academic Press, New York, 1971.


\bibitem{Pross}
Pr\"ossdorf, S.,  {\em Some Classes of Singular Equations}. North-Holland, Amsterdam, 1978.


\bibitem{R}
Riesz, M.,  Sur les fonctions conjuguées. {\em Math. Z.}, 27 (1927) 2, 218-244.

\bibitem{W}
Widom, H., 
Singular integral equations in Lp.
{\em Trans. Amer. Math. Soc.}, 97 (1960),  131-160.

\end{thebibliography}
\end{document}